\documentclass[11pt]{amsart}

\usepackage{latexsym,amssymb,amsfonts,amsmath,mathrsfs,bm}
\newcommand{\sumstar}{\sideset{}{^{*}}\sum}
\newcommand{\sumflat}{\sideset{}{^{\flat}}\sum}
\newcommand{\sumsharp}{\sideset{}{^{\sharp}}\sum}
\newcommand{\chibar}{\overline{\chi}}
\renewcommand{\AA}{{\mathcal{A}}}
\newcommand{\BB}{{\mathcal{B}}}
\renewcommand{\SS}{{\mathcal{S}}}

\newcommand{\R}{{\mathbb R}}

\newcommand{\Z}{{\mathbb Z}}
\newcommand{\eps}{{\varepsilon}}

\newtheorem{theorem}{Theorem}[section]

\newtheorem*{theorem25}{Theorem 2.5}
\newtheorem{proposition}[theorem]{Proposition}
\newtheorem{corollary}[theorem]{Corollary}

\newtheorem{lemma}[theorem]{Lemma}

\theoremstyle{definition}

\theoremstyle{remark}

\newtheorem*{remark1}{Remark}
\newtheorem*{convention}{Convention}
\newtheorem*{note}{Note}

\numberwithin{equation}{section}

\begin{document}

\title[Asymptotic Large Sieve]{Asymptotic Large Sieve}
\author{J.B. Conrey, H. Iwaniec and K. Soundararajan}
\begin{address}{American Institute of Mathematics and University of Bristol}
\end{address}
\begin{address}{Rutgers University}
\end{address}
\begin{address}{Stanford University}
\end{address}


\begin{abstract}
Motivated by applications to the study of L-functions, we develop an asymptotic version of the large sieve inequality for 
linear forms in primitive Dirichlet characters. 
\end{abstract}

\maketitle

\section{Introduction}\label{section1}

Before stating our goal and the main results we shall discuss several points which highlight the essence of the
Asymptotic Large Sieve versus the classical Large Sieve Inequality.

Given a sequence of complex numbers $\AA = (a_{n})$ of some significance in analytic number theory we often need to know if
and how much does there exist a cancellation in the sum
$$\sum_{n \leq N}a_{n}$$
which is due to the variation in the arguments of its terms?  For example, if $a_{n} = \mu(n)$ is the M\"{o}bius function
we expect that
$$\sum_{n \leq N}\mu(n) \ll N^{1/2 + \eps}$$
which bound is equivalent to the Riemann Hypothesis for
$$\zeta(s) = \sum_{n}n^{-s} = \prod_{p}(1 - p^{-s})^{-1}.$$
In other words the RH implies that the series
$$\zeta(s)^{-1} = \sum_{n}\mu(n)n^{-s}$$
converges and does not vanish in the complex half plane $\Re{s} > 1/2$.  Today we know weaker estimates, nevertheless very useful
ones, such as
\begin{equation}\label{11}
\sum_{n \leq X}\mu(n) \ll X\exp{(-\sqrt{\log{X}})}.
\end{equation}

It is a popular view that the M\"{o}bius function changes sign quite randomly.  Indeed we expect the twisted sum
$$\sum_{n \leq X}\mu(n)a_{n}$$
with quite general coefficients retains a considerable cancellation, of course with obvious exceptions such like twisting
against itself or cutting the support of $\mu(n)$ in a combinatorial fashion as in the sieve theory.

Loosely speaking the Riemann Hypothesis for $L$-functions ensures cancellation in relevant sums which is in the order of
magnitude equal to the square-root of the number of summation terms.  In analytic number theory this point of view on the RH
is most enlightening because it offers a guideline on to how much cancellation one can rely on when handling sums of
randomly chosen sequences.  The rule of thumb tells us it should be the square-root of the number of terms, but definitely not more.
Indeed the $L$-functions do have zeros on the critical line $\Re{s}=1/2$ which set the limit in question.
How many critical zeros are there is the question which has drawn us in the first place to the development of the ALS in this
paper (see \cite{CIS}).

Although in practice we investigate a specific sequence $\AA = (a_{n})$, its structure is often so complicated that it leaves
us no other option than to regard $a_{n}$ as any given numbers.  How in such a scenario could anyone hope to produce a
cancellation?  Obviously this is not possible for the individual sequence $\AA$.  But when processing a problem at hand in
analytic number theory one applies to $\AA$ a family of suitable ``harmonics'', say $\chi$'s, and the issue of
cancellation extends to the family of twisted sequences $\AA_{\chi} = (a_{n}\chi(n))$.  Think of the Dirichlet idea of
twisting prime numbers by multiplicative characters $\chi \pmod{q}$ in order to capture primes in an arithmetic progression
$p \equiv a \pmod{q}$ with $(a,q)=1$.

In this paper we choose the family of primitive characters not only of a fixed conductor $q$, but more effectively we also
let $q$ vary over a large segment.  Since $\AA = (a_{n})$ is fixed, it cannot be biased to many characters $\chi(n)$ because
the vectors $[\ldots,\chi(n),\ldots]$ are almost orthogonal.  Therefore the twisted sums
\begin{equation}\label{12}
\sum_{n \leq N}a_{n}\chi(n)
\end{equation}
enjoy a considerable cancellation for almost all characters.  Let us call $\chi$ a ``singular'' character for
$\AA = (a_{n})$ if the twisted sum \eqref{12} does not show a significant cancellation.  In real life the number of
singular characters with respect to a given sequence $\AA$ is fixed by the degree of the corresponding zeta function
$$\zeta_{\AA}(s) = \sum_{n}a_{n}n^{-s}.$$
For example, if $\zeta_{\AA}(s) = \zeta(s)^{g}$ then the principal character is the only singular one.  If
$\zeta_{\AA}(s) = \zeta_{K}(s)$ is the zeta function of an abelian number field $K$ of degree $g$, then $\AA$ admits
exactly $g$ singular characters.  If $\zeta_{\AA}(s) = L(s,f)$ or $L(s,\text{sym}^{2}f)$ where $f$ is a Hecke cusp form
on $SL_{2}(\Z)$, then no character is singular.  These series have Euler products of degree $g=2$ or $g=3$ respectively.
Finally if $\zeta_{\AA}(s) = \zeta(s)^{-1}$ then no character is singular.

Using the orthogonality formula
\begin{equation}\label{13}
\frac{1}{\varphi(q)}\sum_{\chi\, (\!\bmod{q})}\chi(m)\chibar(n) = \left\{\begin{array}{rl}1 & \text{if}\,\, m=n \pmod{q}\\
0 & \text{otherwise}\end{array}\right.
\end{equation}
which is valid for $(mn,q)=1$, one quickly shows that for any complex numbers $a_{n}$
$$\sum_{\chi \,(\!\bmod{q})}\Big|\sum_{n \leq N}a_{n}\chi(n)\Big|^{2} \leq (q+N)\sum_{n\leq N}|a_{n}|^{2}.$$
This easy result is quite interesting, it exhibits the law of a square-root of cancellation as soon as the family of
characters is large enough.  Indeed, by Cauchy's inequality applied to every single sum one gets
$$\Big|\sum_{n \leq N}a_{n}\chi(n)\Big|^{2} \leq N\sum_{n \leq N}|a_{n}|^{2}$$
whereas the averaging over $\chi \pmod{q}$ shows that the whole factor $N$ in the above trivial estimation can be saved, provided
$q \gg N$.

The classical Large Sieve Inequality for primitive characters reveals a much stronger property of orthogonality, it asserts that
\begin{equation}\label{14}
\sum_{q \leq Q}\frac{q}{\varphi(q)}\sumstar_{\chi\, (\!\bmod{q})}\Big|\sum_{n \leq N}a_{n}\chi(n)\Big|^{2} \leq
(Q^{2}+N)\sum_{n \leq N}|a_{n}|^{2}.
\end{equation}
The superscript $*$ indicates that the summation runs over primitive characters.  Without this restriction, the inequality
\eqref{14} would be obviously false.  Here the saving of factor $N$ occurs much sooner, that is if $Q \gg \sqrt{N}$.
Yes, to achieve this effect one still needs a family of respectful size, however of a lot smaller conductor and this is
the key attraction of the LSI.  The power of the LSI is so huge that it is capable to produce results which can compete with
those obtainable by the Generalised Riemann Hypothesis for the Dirichlet $L$-functions.  Actually the LSI is more versatile
a tool than the GRH, the point being that it works for general sequences $\AA = (a_{n})$ while the latter only for
coefficients of $L$-functions, that is no additional structure of $\AA = (a_{n})$ in \eqref{14} is required.  Moreover,
\eqref{14} also holds for sums restricted to any interval $M < n \leq M+N$ of length $N \geq 1$, regardless where it is
located
\begin{equation}\label{15}
\sum_{q \leq Q}\frac{q}{\varphi(q)}\sumstar_{\chi\, (\!\bmod{q})}\Big|\sum_{M < n \leq M+N}a_{n}\chi(n)\Big|^{2} \leq
(Q^{2}+N)\sum_{M < n \leq M+N}|a_{n}|^{2}
\end{equation}
(cf. Theorem 7.13 of \cite{IK}).  We have also the Hybrid Large Sieve Inequality due to P. X. Gallagher
(cf. Theorem 7.17 of \cite{IK}),
\begin{equation}\label{16}
\sum_{q \leq Q}\,\,\sumstar_{\chi (\!\bmod{q})}\int_{-T}^{T}\Big|\sum_{n \leq N}a_{n}\chi(n)n^{it}\Big|^{2}dt \leq
(Q^{2}T+N)\sum_{n \leq N}|a_{n}|^{2}.
\end{equation}
As a matter of fact \eqref{16} follows from \eqref{15} by smoothing in $t$, squaring out,  and splitting the resulting sum
of $ a_{m}\overline{a_{n}}\chi(m)\chibar(n)(m/n)^{it}$ into short segments.

A slightly more ready to use estimate can be derived from \eqref{14} for bilinear forms of type
\begin{equation}\label{17}
\sum_{q \leq Q}\,\,\sumstar_{\chi\, (\!\bmod{q})}\sum_{m}\sum_{n}a_{m}b_{n}F(m,n)\chi(m)\chibar(n)
\end{equation}
where $\AA = (a_{n})$, $\BB = (b_{n})$ are any sequences of complex numbers and $F(x,y)$ is a nice smooth function supported
in the box $[1,M] \times [1,N]$.  For example, suppose the Mellin transform of $F(x,y)$ satisfies
\begin{equation}\label{18}
(2\pi)^{-2}\iint|\widehat{F}(iu,iv)|dudv \leq \mathcal{L}(M,N).
\end{equation}
Then separating the variables $m,n$ in \eqref{17} by Mellin's inversion and applying Cauchy's inequality one shows
that \eqref{17} is bounded by
\begin{equation}\label{19}
\mathcal{L}(M,N)(Q^{2}+M)^{1/2}(Q^{2}+N)^{1/2}\Big(\sum_{n \leq M}|a_{n}|^{2}\Big)^{1/2}\Big(
\sum_{n \leq N}|b_{n}|^{2}\Big)^{1/2}.
\end{equation}
The separation of variables is not expensive since \eqref{18} is often quite small
$$\mathcal{L}(M,N) \ll (\log{2M})(\log{2N}).$$

The terminology ``Large Sieve'' has been used in the literature for seventy years, originating from Yu. V. Linnik's work \cite{Lin}.
Subsequently the method has been modified and generalized to the extent that the trace of sieve genes is no longer
recognizable.  Our goal is to develop the ``Asymptotic Large Sieve'' which gives a more precise asymptotic formula for
special bilinear forms in place of general upper bounds like \eqref{19}.  At this point we feel it is appropriate to call any
result of type \eqref{14} a ``Large Sieve Inequality'' and we are going to explain how this extra word ``Inequality'' addresses
the essence of the results.

Yes, the upper bound \eqref{14} is almost best possible, but it is not perfect; for one reason the number of characters in
the family is smaller than $Q^{2}$ by a constant factor.  Some applications are sensitive to constant factors, so losing the
true cardinality of the family is not acceptable.  On the other hand there is no chance to turn the LSI into an asymptotic
formula by refining standard arguments.  Every known approach makes an appeal in one way or another to the duality
principle for bilinear forms, hence no chance to maintain asymptotic values.  This seemingly little sacrifice, nevertheless makes
the duality arguments so powerful.

To get a better understanding of what is at stake, let us examine the Large Sieve Inequality with additive characters
\begin{equation}\label{110}
\sum_{q \leq Q}\,\,\sumstar_{a\, (\!\bmod{q})}\Big|\sum_{n \leq N}a_{n}e\Big(\frac{an}{q}\Big)\Big|^{2} \leq
(Q^{2}+N)\sum_{n \leq N}|a_{n}|^{2}.
\end{equation}
By the way \eqref{14} follows from \eqref{110} by using Gauss sums.  Here the dual form of \eqref{110} asserts that
\begin{equation}\label{111}
\sum_{n\leq N}\Big|\sum_{q \leq Q}\,\,\sumstar_{a\, (\!\bmod{q})}\gamma_{a/q}e\Big(\frac{an}{q}\Big)\Big|^{2} \leq
(Q^{2}+N)\sum_{q \leq Q}\,\,\sumstar_{a\,(\!\bmod{q})}|\gamma_{a/q}|^{2}.
\end{equation}
for any complex numbers $\gamma_{a/q}$.  Be aware that this equivalence requires testing \eqref{110} for all vectors
$[\ldots,a_{n},\ldots]$ versus \eqref{111} for all vectors $[\ldots,\gamma_{a/q},\ldots]$, not just a few chosen ones.
Therefore, during the passage through duality one has no access to special features of the sequence $\AA = (a_{n})$ which
might be usable, but are discarded.

Next, due to positivity, the dual from \eqref{111} can be smoothed in $n$ before squaring out and changing the order of
summation.  The benefit of such procedure is that one can execute the summation over integers $n$ quite precisely, because all integers
in an interval are evenly spaced.  If one treats \eqref{110} directly then a problem occurs with the distribution of the rational
points $a/q$ modulo one.  These points are well-spaced (the consecutive gaps are $\geq Q^{-2}$), but not evenly.  Therefore
some levelling and smoothing is necessary, which operations create the source where the cardinality factor is lost.
Had one decided to control the gaps between the points $a/q$, then the most adequate tool for the job would have been the spectral
analysis on the modular surface, which is in a different league than that on $\R$.  Fortunately a little use of the
M\"{o}bius function does the job adequately.

As explained above we have to quit the robust duality ideas and start treating the bilinear form \eqref{17} directly.
Therefore, from the very beginning the positivity aspect does not rule the show.  The success comes at the price that the
coefficients $a_{n},b_{n}$ are no longer arbitrary.  Our conditions will be imposed gradually to keep the intermediate
results as general as we can.

The results of this paper are used in [CIS] to obtain lower bounds for the proportion of simple zeros on the critical line of
families of twists of $GL_1$, $GL_2$ and $GL_3$ L-functions.

\vspace{.1in}
ACKNOWLEDGEMENTS. These works were begun at AIM in 1998 and continued over the years at AIM, Rutgers, IAS, Stanford, Bristol, and MSRI. We gratefully acknowledge the support of all of these institutions. This work was also supported in part by grants from the National Science Foundation.

\section{General Results}\label{section2}

Our objective is to evaluate asymptotically the bilinear form
\begin{equation}\label{21}
\SS(\AA \times \BB) = \sum_{q}\frac{\Psi(q/Q)}{\varphi(q)}\sumstar_{\chi \,(\!\bmod{q})}\sum_{m}\sum_{n}a_{m}b_{n}
F(m,n)\chi(m)\chibar(n)
\end{equation}
for two sequences $\AA = (a_{m})$, $\BB = (b_{n})$ of complex numbers which will be specialized in due course.  Here the conductor
$q$ is restricted by a smooth function $\Psi$ of compact support in $\R^{+}$, so $q\asymp Q$, and $Q$ is a large parameter.
Such a smooth cut-off makes no difference in applications, yet it helps in technical arguments.

The function $F(x,y)$ is assumed to be smooth and supported in a square box
\begin{equation}\label{22}
1 \leq x,y \leq N, \qquad N \geq 2.
\end{equation}
We assume that
\begin{equation}\label{23}
x^{i}y^{j}|F^{(i,j)}(x,y)| \leq 1, \qquad \text{if}\,\, 0\leq i,j \leq 2.
\end{equation}
Sometimes we need the variables $m,n$ to run independently, so they need to be separated in $F(m,n)$.  To this end one
can use the Mellin inversion technique, that is we write
$$F(x,y) = \oint\!\!\oint \widehat{F}(u,v)x^{-u}y^{-v}dudv$$
with
$$\widehat{F}(u,v) = \iint F(x,y)x^{u-1}y^{v-1}dxdy.$$
We may only need the Mellin transform $\widehat{F}$ on the purely imaginary lines.  Then \eqref{23} yields (by partial
integration)
$$|\widehat{F}(iu,iv)| \leq \frac{(2\log{N})^{2}}{(1+u^{2})(1+v^{2})}$$
for $u,v$ real, hence
$$\iint|\widehat{F}(iu,iv)|dudv \leq (2\pi\log{N})^{2}.$$
We shall make some modifications of $F(x,y)$ in Section \ref{section9}.

Next we put some growth conditions for the coefficients.  We assume
\begin{equation}\label{24}
|a_{m}| \leq m^{-1/2}\tau(m)^{A}, \qquad |b_{n}| \leq n^{-1/2}\tau(n)^{A}.
\end{equation}
Note that these conditions do not limit very much the applicability of results.  If one has apparently weaker conditions
$a_{m} \ll m^{\eps - 1/2}$, $b_{n} \ll n^{\eps - 1/2}$, then a re-normalization by factor $N^{\eps}$ brings them down
to \eqref{24} while the error term in the obtained asymptotic formula gets worse by the factor $N^{\eps}$.

In applications our coefficients appear as the convolution
\begin{equation}\label{25}
a_{m} = \frac{1}{\sqrt{m}}\sum_{lr=m}\lambda(l)\rho(r)
\end{equation}
with $\lambda(l)$ being the coefficients of an $L$-function
\begin{equation}\label{26}
\mathcal{L}(s) = \sum_{l}\lambda(l)l^{-s}
\end{equation}
and $\rho(r)$ being the coefficients of $\mathcal{L}(s)^{-1}$ modified by  smooth weights  supported on
\begin{equation}\label{27}
1 \leq r \leq X.
\end{equation}
In this example
$$\zeta_{\AA}(s) = \sum_{m}a_{m}m^{-s} = \mathcal{L}(s + \frac{1}{2})\mathcal{M}(s + \frac{1}{2})$$
where
$$\mathcal{M}(s) = \sum_{1 \leq r \leq X}\rho(r)r^{-s}$$
is a Dirichlet polynomial called a ``mollifier''.  We shall succeed with $\mathcal{L}(s)$ having the Euler product of
degree $g \leq 3$ (barely missing degree four).  Therefore $|\lambda(l)| \leq \tau_{g}(l)$, $|\rho(r)| \leq \tau_{g}(r)$,
$|a_{m}| \leq m^{-1/2}\tau(m)\tau_{g}(m) \leq m^{-1/2}\tau(m)^{g}$, hence \eqref{24} holds with $A = g$.

There is plenty of room for modifications in the above setting.  Let us illustrate some of these.  At some point we shall need
results for coefficients \eqref{25} where the generating function \eqref{26} is slightly shifted to $\mathcal{L}(s + \alpha)$
with a small complex number $\alpha$,
\begin{equation}\label{28}
|\alpha| \ll (\log{Q})^{-1}.
\end{equation}
This can easily be covered by the case $\alpha = 0$.  To this end, multiply \eqref{25} throughout by $m^{-\alpha}$ and use
$r^{-\alpha}\rho(r)$, $m^{\alpha}F(m,n)$ in place of $\rho(r)$, $F(m,n)$.  The same can be done for the coefficients $b_{n}$.

However, we are going to work for a while with free coefficients $a_{m},b_{n}$ satisfying only \eqref{24} until the structure
\eqref{25} becomes necessary.  In this somewhat general presentation we obtain prefabricated products some of which can be taken
for further developments in the future.

For all positive integers $m,n$ we put
\begin{equation}\label{29}
\Delta(m,n) = \sum_{q}\frac{\Psi(q/Q)}{\varphi(q)}\sumstar_{\chi \,(\!\bmod{q})}\chi(m)\chibar(n).
\end{equation}
This is an averaging operator over our family of harmonics (which are the primitive characters of conductor of size $Q$).
There is nothing special with our choice of arithmetic weight $1/ \varphi(q)$; other natural choices would be
$1/ \varphi^{*}(q)$ where $\varphi^{*}(q)$ denotes the number of primitive characters modulo $q$ ($\varphi^{*} = \mu \ast \varphi$).
This is less convenient, because $\varphi^{*}(q) = 0$ if $2 \| q$.

The bilinear form \eqref{21} becomes
\begin{equation}\label{210}
\SS(\AA \times \BB) = \sum_{m}\sum_{n}a_{m}b_{n}F(m,n)\Delta(m,n).
\end{equation}
A substantial part of our work on $\SS(\AA \times \BB)$ concerns $\Delta(m,n)$ acting on individual $m,n$.  Along the lines
one can find various decompositions of $\Delta(m,n)$ whose constituents emerge from meaningful sources.  Some of these are
matured and simple, but unfortunately we have to state some important ones in a crude form before giving final estimations.
In particular the exact expression \eqref{60} is hard to grasp, but we do not want to settle with just estimates which would close the
door for future exploration.

The hardest job is to excavate the leading terms.  In order to find it a bit easier we shall accept conditions on $\AA, \BB$ such
that only the contribution of diagonal terms
\begin{equation}\label{211}
\SS_{diag}(\AA \times \BB) = \sum_{m}a_{m}b_{m}F(m,m)\Delta(m,m)
\end{equation}
makes the leading term.  Here we have
\begin{equation}\label{212}
\Delta(m,m) = \delta(m)\Delta(1,1) + O(\tau(m)Q^{1/2})
\end{equation}
with
\begin{equation}\label{213}
\delta(m) = \prod_{p \mid m}\Big(1 - \frac{1}{p}\Big)\Big(1- \frac{1}{p^{2}}-\frac{1}{p^{3}}\Big)^{-1}.
\end{equation}
For
\begin{equation}\label{214}
\Delta(1,1) = \sum_{q}\Psi\Big(\frac{q}{Q}\Big)\frac{\varphi^{*}(q)}{\varphi(q)}
\end{equation}
(the cardinality measure of our family of characters) we get
\begin{equation}\label{215}
\Delta(1,1) = \Psi\mathfrak{S}Q + O(Q^{1/2})
\end{equation}
with
\begin{equation}\label{216}
\Psi = \int\Psi(x)dx,
\end{equation}
\begin{equation}\label{217}
\mathfrak{S} = \prod_{p}\Big(1- \frac{1}{p^{2}}-\frac{1}{p^{3}}\Big).
\end{equation}
Hence
\begin{proposition}
For $a_{m},b_{n}$ and $F(m,n)$ satisfying \eqref{24},\eqref{23}, we have
\begin{equation}\label{218}
\SS_{diag}(\AA \times \BB) = \Psi\mathfrak{S}Q\sum_{m}a_{m}b_{m}\delta(m)F(m,m) + O(Q^{1/2}(\log{N})^{A}).
\end{equation}
\end{proposition}

\begin{convention}
The exponent $A$ in \eqref{218} depends on that in the growth conditions \eqref{24}.  For notational convenience we are going
to use $A$ as an exponent which is not necessarily the same in each occurence, it is allowed to depend on foregoing acceptable
constants.
\end{convention}

Besides the primary terms of the diagonal, a secondary source for contribution to the main term is not so obvious as the diagonal one;
it rests in narrow strips parallel to the diagonal.  A substantial contribution may come out of the terms
$a_{m}b_{n}F(m,n)$ with $|m-n| \asymp Q$, but not from strips of much smaller width.  Hence we shall get quickly

\begin{theorem}\label{theorem22}
For any complex numbers $a_{m},b_{n}$ satisfying \eqref{24} and $F(m,n)$
satisfying \eqref{23} with $N \leq Q^{1-\eps}$ we have
\begin{equation}\label{219}
\SS(\AA\times\BB) = \SS_{diag}(\AA\times\BB) + O(Q(\log{Q})^{-C})
\end{equation}
for any $C >0$, the implied constant depends only on $\eps$, $C$ and $A$ in \eqref{24}.
\end{theorem}
If $N$ is much larger than $Q$, then one is faced with the problem of asymptotic evaluation of double sums
$\sum\sum a_{m}b_{n}F(m,n)$ with relatively small $m-n$, and it is difficult to grasp such skinny domains when counting
with most general coefficients $a_{m},b_{n}$.
These off-diagonal sums, which may yield the secondary main term, result from switching moduli, hence the plain sum
without twists by characters emerged here despite that the principal character is absent in our original family
(the completeness in arithmetic comes to prominence!)

Our goal is to show that \eqref{219} holds for $N \leq Q^{2-\eps}$, but subject to some conditions on the coefficients $a_{m}$,$b_{n}$.
\begin{theorem}\label{theorem23}
Let $\AA = (a_{m})$ be given by \eqref{25} with any complex numbers $\rho(r)$ for $1 \leq r \leq X$ satisfying
\begin{equation}\label{220}
|\rho(r)| \leq \tau(r)^{A}
\end{equation}
for some constant $A \geq 0$.  Moreover, suppose that the $L$-function \eqref{26} has Euler product of degree $g \leq 3$, and no character
is singular for $\mathcal{L}(s)$, i.e.
\begin{equation}\label{221}
\mathcal{L}(s,\chi) = \sum_{1}^{\infty}\lambda(l)\chi(l)l^{-s}
\end{equation}
is entire for all $\chi$.  Assume similar conditions for $\BB = (b_{n})$.    Then \eqref{219} holds if $N \leq Q^{2-\eps}$ and
\begin{equation}\label{222}
X \leq Q^{1-\eps} \quad \text{if}\,\,\,g=1,2,
\end{equation}
\begin{equation}\label{223}
X \leq Q^{1/2} \quad \text{if}\,\,\,g=3.
\end{equation}
\end{theorem}

\begin{note}
It needs to be said that by an $L$-function we mean one of those whose twists by primitive characters satisfy
proper functional equation, see Section 5.1 of \cite{IK}.
\end{note}

\begin{remark1}
The case $g=4$ could be also covered by our arguments giving \eqref{219} if $X \leq Q^{\delta}$,
however it is not interesting because our condition $N \leq Q^{2-\eps}$ does not go far enough for applications.
The problem is that \eqref{221} of degree $g=4$ can be approximated by two partial sums of length $N \asymp Q^{2}$, but not shorter.
\end{remark1}

Actually the case of degree $g=1$ is void in the statement of Theorem \ref{theorem23}.
Indeed in this case $\mathcal{L}(s,\chi)$ is just a Dirichlet $L$-function, and for $\chi = 1$ it becomes
$\mathcal{L}(s) = \zeta(s)$ which has a simple pole at $s=1$.  In order to cover this case we make an extra cancellation
condition for the factor sequence $\rho(r)$;
\begin{equation}\label{224}
\sum_{r \leq y}\rho(dr) \ll \tau(d)y(\log{y})^{-C}, \quad \text{if}\,\,y \geq 2
\end{equation}
with any $C \geq 0$, the implied constant depending on $C$.

\begin{theorem}\label{theorem24}
Let $\AA = (a_{m})$, $\BB = (b_{n})$ be sequences given by \eqref{25} satisfying the conditions of
Proposition \ref{proposition81}.  Then \eqref{219} holds if $X \leq Q^{1-\eps}$ and $N \leq Q^{2-\delta}$.
\end{theorem}

It is obvious what factors $\rho(r)$ are in our mind when creating the condition \eqref{224}.
Since in applications the $\rho(r)$ come as coefficients in a mollifier of $\zeta(s)$, it takes the form
\begin{equation}\label{225}
\rho(r) = \mu(r)w(r)
\end{equation}
where $w(r)$ is a smooth function supported on $1 \leq r \leq X$ with
\begin{equation}\label{226}
|w(r)| \leq 1, \qquad r|w'(r)| \leq 1.
\end{equation}
In this case the cancellation condition \eqref{224} follows from \eqref{11}.

Of course, there are other choices for $\rho(r)$ satisfying \eqref{224}, for
example the Fourier coefficients of a cusp form.

\section{Decomposition of $\Delta(m,n)$}\label{section3}

The proofs of theorems from Section \ref{section2} go in several steps.  We begin by decomposing $\Delta(m,n)$
into distinct parts and split the bilinear form $\SS(\AA\times\BB)$ accordingly.

If $(mn,q)=1$ then
$$\sumstar_{\chi \,(\!\bmod{q})}\chi(m)\chibar(n) = \sum_{\substack{cd=q\\d \mid (m-n)}}\mu(c)\varphi(d).$$
Hence for any $m,n$ we have
$$\Delta(m,n) = \sum\sum_{\substack{(cd,mn)=1\\d\mid(m-n)}}\Psi\Big(\frac{cd}{Q}\Big)
\frac{\mu(c)\varphi(d)}{\varphi(cd)}.$$
We split this into
\begin{equation}\label{31}
\Delta(m,n) = \Delta'(m,n) + \Delta''(m,n),
\end{equation}
say, where $\Delta'(m,n)$ is the double sum over $c$,$d$ restricted by $c \leq C$ and
$\Delta''(m,n)$ is the complimentary sum, which is restricted by $c > C$.  Here $C$ is at our disposal
and it will be chosen later as a small power of $Q$.

We continue decomposing $\Delta''(m,n)$.  First notice that $cd \asymp Q$,
hence for $c > C$ we have $d \ll QC^{-1}$.  Since $d$ is relatively small we reverse the above transformation, that is
we return to characters detecting the congruence $m \equiv n \pmod{d}$ by the
orthogonality formula \eqref{13}.
We write
$$\frac{1}{\varphi(d)}\sum_{\chi \,(\!\bmod{d})}\chi(m)\chibar(n) =
\frac{1}{\varphi(d)}\sum_{\substack{kl=d\\(l,mn)=1}}\sumstar_{\chi \,(\!\bmod{k})}\chi(m)\chibar(n).$$
On the left side $\chi$ runs over all characters of modulus $d$ whereas on the right side $\chi$ runs over the primitive characters of conductor $k$ for every
$k \mid d$.  The complimentary factor $l$ appears as a free variable (well, almost free apart from minor arithmetical
obstacles).  Now we get
$$\Delta''(m,n) = \sum_{\substack{c > C\\(cl,mn)=1}}\mu(c)\sum_{l}\sum_{k \leq K}
\frac{\Psi(ckl/Q)}{\varphi(ckl)}\sumstar_{\chi \,(\!\bmod{k})}\chi(m)\chibar(n)$$
where
\begin{equation}\label{32}
K \asymp QC^{-1}.
\end{equation}
The restriction $k\le K$ in $\Delta''(m,n)$ with $K$ of size \eqref{32} is redundant with the support of $\Psi$ because $c > C$
and $l \geq 1$, nevertheless we keep it for the record.  At some point later we shall not be able to track the restriction
$k \leq K$ from the support of $\Psi$ when we lose sight of it in Fourier analysis.
Note that $cl \geq C > 1$, so by M\"{o}bius inversion we can switch the range $c > C$ to $c \leq C$ getting

\begin{lemma}\label{lemma31}
For every $m,n \geq 1$ we have
\begin{equation}\label{33}
\Delta''(m,n) = -\sum_{\substack{c \leq C\\(cl,mn)=1}}\mu(c)\sum_{l}\sum_{k \leq K}
\frac{\Psi(ckl/Q)}{\varphi(ckl)}\sumstar_{\chi \,(\!\bmod{k})}\chi(m)\chibar(n).
\end{equation}
\end{lemma}

It is in the above transformation where properties of the M\"{o}bius function are used with manifestation that
our original bilinear form $\SS(\AA\times\BB)$ goes over the primitive characters rather than over all characters.

What we accomplished in the expression \eqref{33} is that the involved characters have conductor $k \leq K$ significantly smaller than $q \asymp Q$ in
$\SS(\AA\times\BB)$.  We postpone further transformations of $\Delta''(m,n)$ to the next section.

Now we are going to perform transformations of $\Delta'(m,n)$, although quite differently than in $\Delta''(m,n)$, but with similar objective to reduce the size of the conductor of involved characters.  Recall that
\begin{equation}\label{34}
\Delta'(m,n) = \sum\sum_{\substack{(cd,mn)=1\\c\leq C, d \mid (m-n)}}\Psi\Big(\frac{cd}{Q}\Big)\frac{\mu(c)\varphi(d)}{\varphi(cd)}.
\end{equation}
For $m=n$ (the diagonal terms in $\SS(\AA\times\BB)$) the condition $d \mid (m-n)$
holds automatically and $\Delta'(m,n)$ is close to the full sum
\begin{eqnarray}\label{35}
\Delta(m,n) &=&\sum\sum_{(cd,m)=1}\Psi\Big(\frac{cd}{Q}\Big)\frac{\mu(c)\varphi(d)}{\varphi(cd)}
\\\nonumber &=& \sum_{(q,m)=1}\Psi\Big(\frac{q}{Q}\Big)\frac{\varphi^{*}(q)}{\varphi(q)}.
\end{eqnarray}
The difference is estimated trivially by
$$\sum_{c > C}\frac{1}{\varphi(c)}\sum_{d}|\Psi\Big(\frac{cd}{Q}\Big)| \ll
\sum_{c > C}\frac{Q}{c\varphi(c)} \ll \frac{Q}{C}.$$
\begin{lemma}\label{lemma32}
For every $m \geq 1$ we have
\begin{equation}\label{36}
\Delta'(m,m) = \Delta(m,m) + O(QC^{-1}).
\end{equation}
\end{lemma}

Now let $m\neq n$.  Note that $d$ in \eqref{34} is pretty large, $d \asymp QC^{-1}$.  Our intention is to switch
the divisor $d$ of $|m-n|$ to the complementary one which is much smaller.  This requires a few steps to relax the minor arithmetical obstacles.
First writing
\begin{equation}\label{37}
\frac{\varphi(d)}{\varphi(cd)} = \frac{1}{\varphi(c)}\sum_{a \mid (c,d)}\frac{\mu(a)}{a}
\end{equation}
we get
$$\Delta'(m,n) = \sum\sum_{\substack{(ac,mn)=1\\ ac\leq C}}\frac{\mu(a)\mu(ac)}{a\varphi(ac)}
\sum_{\substack{(d,mn)=1\\ad \mid (m-n)}}\Psi\Big(\frac{a^{2}cd}{Q}\Big).$$
Next we remove the condition $(d,mn)=1$ by M\"{o}bius formula getting
$$\sum_{\substack{(d,mn)=1\\ad \mid (m-n)}}\Psi\Big(\frac{a^{2}cd}{Q}\Big)
= \sum_{g \mid (m,n)}\mu(g)\sum_{adg\mid(m-n)}\Psi\Big(\frac{a^{2}cdg}{Q}\Big).$$
Here $d$ is a free divisor of $|m-n|/ag$, so we can switch $d$ to its complimentary divisor, say $b$, getting
$$\Delta'(m,n) = \sum\sum_{\substack{(ac,mn)=1\\ac \leq C}}\frac{\mu(a)\mu(ac)}{a\varphi(ac)}
\sum\sum_{\substack{g \mid (m,n)\\abg \mid (m-n)}}\mu(g)\Psi\Big(\frac{ac}{bQ}|m-n|\Big).$$
Pull out the common factor $(b,m/g) = (b,n/g) = h$, say, getting
$$\Delta'(m,n) = \sum\sum_{\substack{(ac,mn)=1\\ac \leq C}}\frac{\mu(a)\mu(ac)}{a\varphi(ac)}
\sum\sum_{gh \mid (m,n)}\mu(g)\sum_{\substack{ab\mid(\mathfrak{m}-\mathfrak{n})\\
(b,\mathfrak{m}\mathfrak{n})=1}}\Psi\Big(\frac{ac}{bhQ}|m-n|\Big)$$
where here and thereafter we put for notational simplicity
\begin{equation}\label{38}
m = gh\mathfrak{m}, \qquad n = gh\mathfrak{n}.
\end{equation}
Since $(ab, \mathfrak{m}\mathfrak{n}) = 1$ we detect the divisibility $ab \mid (\mathfrak{m} - \mathfrak{n})$ by
characters $\chi \pmod{ab}$ getting
\begin{equation*}\begin{split}
\Delta'(m,n) = &\sum\sum_{\substack{(ac,mn)=1\\ac \leq C}}\frac{\mu(a)\mu(ac)}{a\varphi(ac)} \sum\sum_{gh \mid (m,n)}\mu(g)\\
&\sum_{b}\frac{1}{\varphi(ab)}\sum_{\chi \,(\!\bmod{ab})}\chi(\mathfrak{m})\chibar(\mathfrak{n})
\Psi\Big(\frac{ac}{bhQ}|m-n|\Big).
\end{split}\end{equation*}
Finally we express this in terms of primitive characters as follows.
\begin{lemma}\label{lemma33}
For every $m,n \geq 1$, $m \neq n$ we have
\begin{equation}\begin{split}\label{39}
\Delta'(m,n) = &\sum\sum_{\substack{(ac,mn)=1\\ac \leq C}}\frac{\mu(a)\mu(ac)}{a\varphi(ac)} \sum\sum_{gh \mid (m,n)}\mu(g)\\
&\sum\sum_{\substack{kl \equiv 0 \,(\!\bmod{a})\\(l,\mathfrak{m}\mathfrak{n})=1}}\frac{1}{\varphi(kl)}
\sumstar_{\chi \,(\!\bmod{k})}\chi(\mathfrak{m})\chibar(\mathfrak{n})
\Psi\Big(\frac{ac}{klhQ}|m-n|\Big).
\end{split}\end{equation}
\end{lemma}

\begin{remark1}
The right side of \eqref{39} vanishes if $m=n$, so we do not need to remember that $m \neq n$ when inserting this expression in to the bilinear form
$\SS(\AA\times\BB)$.  The characters in \eqref{39} have conductor $k \asymp ac|m-n|/lhQ \ll CNQ^{-1}$ which is significantly smaller than $q \asymp Q$ in
$\SS(\AA\times\BB)$.  The variable $l$ is essentially free, it emerges in the transition from arbitrary characters to the primitive ones.
Although the source of $l$ seems to be technical, nevertheless this variable features in forthcoming transformations.
\end{remark1}

\section{The Euler-Maclaurin Summation}\label{section4}

Our next step will be to execute the summation over the variable $l$ in \eqref{33} and \eqref{39}.  Since $l$ is not yet entirely free we are going to
state formulas which liberate $l$ from all relevant constraints.

\begin{lemma}\label{lemma41}
For $(a,s) = 1$ and $l \geq 1$ we have
\begin{equation}\label{41}
\sumflat_{\substack{u \mid l\\(u,a)=1}}\mu((u,s))\frac{\varphi((u,s))}{\varphi(u)} = \left\{\begin{array}{ll}\varphi(a)l/\varphi(al)& \quad \text{if}
\,\, (l,s)=1\\0& \quad \text{otherwise}\end{array}\right.
\end{equation}
where the superscript $\flat$ restricts the summation to squarefree numbers.
\end{lemma}
\begin{proof}
It suffices to prove \eqref{41} for $l$ prime, in which case the result is easy to check.
\end{proof}
\begin{lemma}\label{lemma42}
For $m,n,u \geq 1$, $u$ squarefree, we have
\begin{equation}\label{42}
\mu((u,mn))\varphi((u,mn)) = \sum\sum\sum_{\substack{\alpha\beta\gamma \mid u\\ \alpha\beta\mid m, \alpha\gamma \mid n}}
\alpha\beta\gamma\mu(\beta\gamma).
\end{equation}
\end{lemma}
\begin{proof}
It suffices to prove \eqref{42} for $u$ prime, in which case the result is easy to check.
\end{proof}

Applying \eqref{41} to \eqref{33} we obtain
\begin{corollary}\label{corollary43}
For every $m,n \geq 1$ we have
\begin{equation}\begin{split}\label{43}
\Delta''(m,n) = & -\sum_{\substack{c \leq C\\(c,mn)=1}}\sum_{k \leq K}\frac{\mu(c)}{\varphi(ck)}\sumflat_{(u,ck)=1}\mu((u,mn))
\frac{\varphi((u,mn))}{u\varphi(u)}\\
& \sum_{l}l^{-1}\Psi\Big(\frac{ckul}{Q}\Big)\sumstar_{\chi \,(\!\bmod{k})}\chi(m)\chibar(n).
\end{split}\end{equation}
\end{corollary}

Similarly we apply \eqref{41} to \eqref{39}.  We have $l = l'a/(a,k)$ in \eqref{39}, where $l'$ runs over positive integers with
$(l',\mathfrak{m}\mathfrak{n}) = 1$ and $\varphi(kl) = \varphi(l'[a,k])$.  Hence \eqref{41} with $a,s,l$ replaced by $[a,k]$,
$\mathfrak{m}\mathfrak{n}$, $l'$ respectively, yields
\begin{corollary}\label{corollary44}
For every $m,n \geq 1$, $m \neq n$ we have (recall the notation \eqref{38})
\begin{equation}\begin{split}\label{44}
\Delta'(m,n) =& \sum\sum_{\substack{(ac,mn)=1\\ac \leq C}}\frac{\mu(a)\mu(ac)}{a\varphi(ac)}\sum\sum_{gh \mid (m,n)}\mu(g)
\sum_{k}\frac{1}{\varphi([a,k])}\\
&\sumflat_{(u,ak)=1}\mu((u,\mathfrak{m}\mathfrak{n}))\frac{\varphi((u,\mathfrak{m}\mathfrak{n}))}{u\varphi(u)}\sum_{l}l^{-1}
\Psi\Big(\frac{(a,k)c|m-n|}{kulhQ}\Big)\sumstar_{\chi \,(\!\bmod{k})}\chi(\mathfrak{m})\chibar(\mathfrak{n}).
\end{split}\end{equation}
\end{corollary}

From the support of $\Psi$ in \eqref{44} it follows that $k \ll CNQ^{-1}$.  Assuming
\begin{equation}\label{45}
N \ll Q^{2}C^{-2}
\end{equation}
we find that $k \leq K$ in \eqref{44}, and we keep this redundant restriction in the forthcoming expressions.

Next we execute the free summation over $l$ in \eqref{43} and \eqref{44} by the Euler-MacLaurin formula
\begin{equation}\label{46}
\sum_{l=1}^{\infty}F(l) = \int_{0}^{\infty}[F(t) + \{t\}F'(t)]dt
\end{equation}
where $\{t\}$ denotes the fractional part of $t$.  All we need about $\{t\}$ is the estimate
$$0 \leq \{t\} \leq \min(1,t).$$
In case of \eqref{43} we use \eqref{46} with $F(t) = t^{-1}\Psi(tT^{-1})$ getting
$$\sum_{l=1}^{\infty}l^{-1}\Psi\Big(\frac{l}{T}\Big) = \widehat{\Psi}(0) + T^{-1}\Psi_{2}(T), \qquad \text{for}\,\,T=Q/cku$$
where
\begin{equation}\label{47}
\Psi_{2}(T) = \int_{0}^{\infty}(t^{-1}\Psi(t))'\{tT\}dt \ll \min(1,T).
\end{equation}
In case of \eqref{44} we use \eqref{46} with $F(t) = t^{-1}\Psi(Tt^{-1})$ getting
$$\sum_{l=1}^{\infty}l^{-1}\Psi\Big(\frac{T}{l}\Big) = \widehat{\Psi}(0) + \Psi_{1}(T), \qquad \text{for}\,\,T=|\mathfrak{m}-\mathfrak{n}|
(a,k)c/kuhQ$$
where
\begin{equation}\label{48}
\Psi_{1}(T) = \int\Omega(tT)\{t^{-1}\}dt, \qquad \Omega(x) = (x\Psi(x))'.
\end{equation}
Note that in both cases the integral
$$\int F(t)dt = \int\Psi(t)t^{-1} = \widehat{\Psi}(0)$$
takes the same value which does not depend on $T$.

According to \eqref{46} we write
\begin{equation}\begin{split}\label{49}
&\Delta''(m,n) = \widehat{\Psi}(0)\Delta_{0}''(m,n) + \Delta_{2}(m,n)\\
&\Delta'(m,n) = \widehat{\Psi}(0)\Delta_{0}'(m,n) + \Delta_{1}(m,n)
\end{split}\end{equation}
where the terms $\Delta_{0}''(m,n)$, $\Delta_{0}'(m,n)$ are obtained from \eqref{43}, \eqref{44} respectively, with the summation over $l$
being dropped.  We shall see in the following lemma that these terms cancel out.
\begin{lemma}\label{lemma45}
For every $m,n \geq 1$, $m \neq n$, we have
\begin{equation}\label{410}
\Delta_{0}''(m,n) + \Delta_{0}'(m,n) = 0.
\end{equation}
\end{lemma}
\begin{proof}
We have
$$\Delta_{0}''(m,n) = -\sum_{\substack{c \leq C\\(c,mn)=1}}\sum_{k \leq K}\frac{\mu(c)}{\varphi(ck)}\sumflat_{(u,ck)=1}\mu((u,mn))
\frac{\varphi((u,mn))}{u\varphi(u)}\sumstar_{\chi \,(\!\bmod{k})}\chi(m)\chibar(n).$$
On the other hand
\begin{equation*}\begin{split}
\Delta_{0}'(m,n) = &\sum\sum_{\substack{(ac,mn)=1\\ac \leq C}}\frac{\mu(a)\mu(ac)}{a\varphi(ac)}\sum\sum_{gh \mid (m,n)}\mu(g)
\sum_{k\leq K}\frac{1}{\varphi([a,k])}\\
&\sumflat_{(u,ak)=1}\mu((u,\mathfrak{m}\mathfrak{n}))\frac{\varphi((u,\mathfrak{m}\mathfrak{n}))}{u\varphi(u)}
\sumstar_{\chi \,(\!\bmod{k})}\chi(\mathfrak{m})\chibar(\mathfrak{n}).
\end{split}\end{equation*}
By M\"{o}bius formula $gh=1$, so $\mathfrak{m} = m$, $\mathfrak{n} = n$.  Put $ac=b$, so $(b,mn)=1$, $b \leq C$.  The sum over $u$ with
$(u,ak)=1$ is equal to the same sum over $u$ with $(u,bk)=1$ times the same sum over $u \mid c$ with $(u,k)=1$, which is
$$\sum_{\substack{u \mid c\\(u,k)=1}}\mu((u,mn))\frac{\varphi((u,mn))}{u\varphi(u)} = \prod_{\substack{p\mid c\\p \nmid k}}\Big(1 + \frac{1}{p(p-1)}\Big).$$
Then we have
\begin{eqnarray*}
\sum_{ac =b}\frac{\mu(a)}{a\varphi([a,k])}\prod_{\substack{p\mid c\\p \nmid k}}\Big(1 + \frac{1}{p(p-1)}\Big) & = &
\frac{1}{\varphi(k)}\sum_{ac =b}\frac{\mu(a)}{a\varphi(a/(a,k))}\prod_{\substack{p\mid c\\p \nmid k}}\Big(1 + \frac{1}{p(p-1)}\Big)\\
& = & \frac{1}{\varphi(k)}\prod_{\substack{p\mid b\\p \nmid k}}\Big(1 - \frac{1}{p}\Big) = \frac{\varphi(b)}{\varphi(bk)}.
\end{eqnarray*}
This shows that the sums $\Delta_{0}'(m,n)$, $\Delta_{0}''(m,n)$  agree except for the sign (to see it clearly re-name $c$ to $b$ in
$\Delta_{0}''(m,n)$).  This completes the proof of Lemma \ref{lemma45}.
\end{proof}

Let us write explicitly the remaining parts of the decompositions \eqref{49}.  These come from the second part of the Euler-MacLaurin formula
\eqref{46}.  We have
\begin{equation}\label{411}\begin{split}
\Delta_{2}(m,n) = -\frac{1}{Q}&\sum_{\substack{c \leq C\\(c,mn)=1}}\mu(c)\sum_{k \leq K}\frac{ck}{\varphi(ck)}
\\
&
\sumflat_{(u,ck)=1}\mu((u,mn))
\frac{\varphi((u,mn))}{\varphi(u)}\Psi_{2}\Big(\frac{Q}{cku}\Big)\sumstar_{\chi \,(\!\bmod{k})}\chi(m)\chibar(n),
\end{split}\end{equation}
\begin{equation}\label{412}\begin{split}
\Delta_{1}(m,n) = &\sum\sum_{\substack{(ac,mn)=1\\ac \leq C}}\frac{\mu(a)\mu(ac)}{a\varphi(ac)}\sum\sum_{gh \mid (m,n)}\mu(g)
\sum_{k \leq K}\frac{1}{\varphi([a,k])}\\
&\sumflat_{(u,ak)=1}\mu((u,\mathfrak{m}\mathfrak{n}))\frac{\varphi((u,\mathfrak{m}\mathfrak{n}))}{u\varphi(u)}
\Psi_{1}\Big(\frac{|m-n|(a,k)c}{kuhQ}\Big)\sumstar_{\chi \,(\!\bmod{k})}\chi(\mathfrak{m})\chibar(\mathfrak{n}).
\end{split}\end{equation}
In the next two sections we shall make further transformations of \eqref{411} and \eqref{412} by applying \eqref{42}.

We shall show that $\Delta_{2}(m,n)$ and $\Delta_{1}(m,n)$ yield small contributions, not for every $m,n$, but due to cancellation in the bilinear
forms
\begin{equation}\label{413}
\SS_{2}(\AA\times\BB) = \sum_{m}\sum_{n}a_{m}b_{n}F(m,n)\Delta_{2}(m,n)
\end{equation}
\begin{equation}\label{414}
\SS_{1}(\AA\times\BB) = \sum_{m}\sum_{n \neq m}a_{m}b_{n}F(m,n)\Delta_{1}(m,n).
\end{equation}
Estimation of $\SS_{2}(\AA\times\BB)$ is given in Section \ref{section5} by a straightforward application of the large sieve inequality.  The
bilinear form $\SS_{1}(\AA\times\BB)$ will be treated by more subtle arguments in Sections \ref{section6}, \ref{section7}, \ref{section8}.

We end this section by a further partition of $\Delta_{1}(m,n)$.  To simplify this partition we assume that all characters are non-singular, except
possibly the trivial one $\chi = 1$.  If the trivial character is singular then it takes a special place in \eqref{412}, so we pull out its contribution,
say $\Delta^{+}(m,n)$.  Therefore
\begin{equation}\label{415}
\Delta_{1}(m,n) = \Delta^{+}(m,n) + \Delta^{*}(m,n)
\end{equation}
where $\Delta^{*}(m,n)$ is the sum \eqref{412} with the terms for $k=1$ being omitted, but only if $\chi = 1$ is singular.  Hence
$\Delta^{*}(m,n) = \Delta_{1}(m,n)$ and $\Delta^{+}(m,n) = 0$, unless $\chi = 1$ is singular, in which case
\begin{equation}\label{416}\begin{split}
\Delta^{+}(m,n) = &\sum\sum_{\substack{(ac,mn)=1\\ac \leq C}}\frac{\mu(a)\mu(ac)}{a\varphi(a)\varphi(ac)}\sum\sum_{gh \mid (m,n)}\mu(g)\\
&\sumflat_{(u,a)=1}\mu((u,\mathfrak{m}\mathfrak{n}))\frac{\varphi((u,\mathfrak{m}\mathfrak{n}))}{u\varphi(u)}
\Psi_{1}\Big(\frac{|m-n|ac}{uhQ}\Big).
\end{split}\end{equation}
Accordingly we split the bilinear form
\begin{equation}\label{417}
\SS_{1}(\AA\times\BB) = \SS^{+}(\AA\times\BB) + \SS^{*}(\AA\times\BB).
\end{equation}

If the trivial character is singular then $\Delta^{+}(m,n)$ potentially could yield an extra main term other than the one previously extracted
from the diagonal terms $\Delta(m,m)$.  We shall avoid this possibility due to the extra condition \eqref{224} which kills the problem before
computations become unbearable.

\section{Estimation of $\SS_{2}(\AA\times\BB)$}\label{section5}

First we apply \eqref{42} to \eqref{411} getting
\begin{equation*}\begin{split}
\Delta_{2}(m,n) = &-\frac{1}{Q}\sum_{\substack{c \leq C\\(c,mn)=1}}\mu(c)\sum_{k \leq K}
\sum\sum\sum_{\substack{\alpha\beta\mid m\\\alpha\gamma\mid n}}\mu(\alpha)\mu(\alpha\beta\gamma)\frac{\alpha\beta\gamma ck}{\varphi(\alpha\beta\gamma ck)}\\
&\sumflat_{(u,\alpha\beta\gamma ck)=1}\varphi(u)^{-1}\Psi_{2}\Big(\frac{Q}{\alpha\beta\gamma cku}\Big)\sumstar_{\chi \,(\!\bmod{k})}\chi(m)\chibar(n).
\end{split}\end{equation*}
Hence
\begin{equation*}\begin{split}
\SS_{2}(\AA\times\BB) = &-\frac{1}{Q}\sum_{\substack{c \leq C\\(c,mn)=1}}\mu(c)\sum_{k \leq K}
\sum_{\alpha}\sum_{\beta}\sum_{\gamma}\mu(\alpha)\mu(\alpha\beta\gamma)\frac{\alpha\beta\gamma ck}{\varphi(\alpha\beta\gamma ck)}\\
&\sumflat_{(u,\alpha\beta\gamma ck)=1}\varphi(u)^{-1}\Psi_{2}\Big(\frac{Q}{\alpha\beta\gamma cku}\Big)\sumstar_{\chi \,(\!\bmod{k})}
\sum\sum_{\substack{m \equiv 0 (\alpha\beta)\\n \equiv 0 (\alpha\gamma)}}a_{m}b_{n}\chi(m)\chibar(n)F(m,n).
\end{split}\end{equation*}
By \eqref{47} we have $\Psi_{2}(Q/\alpha\beta\gamma cku) \ll \min(1/Q\alpha\beta\gamma cku)$.  Moreover we have
$$\sum_{u}\varphi(u)^{-1}\min(1,Q/\alpha\beta\gamma ku) \ll \min(1,Q/\alpha\beta\gamma k)\log{Q}.$$
Hence we obtain
\begin{equation*}\begin{split}
\SS_{2}(\AA\times\BB) \ll &(\log{Q})^{2}\frac{C}{Q}\sum_{k \leq K}
\sum_{\alpha}\sum_{\beta}\sum_{\gamma}\min(1,Q/\alpha\beta\gamma k)\\&\sumstar_{\chi \,(\!\bmod{k})}
\Big|\sum_{m \equiv 0 (\alpha\beta)}\sum_{n \equiv 0 (\alpha\gamma)}a_{m}b_{n}\chi(m)\chibar(n)F(m,n)\Big|.
\end{split}\end{equation*}
Recall \eqref{24} and \eqref{32}.  Then by the large sieve inequality we derive (see \eqref{17} and \eqref{19})
\begin{equation*}\begin{split}
\SS_{2}(\AA\times\BB) \ll (\log{Q})^{2A+2}K^{-1}\sum\sum\sum_{\alpha\beta\gamma \leq N^{2}}&\min\Big(1,\frac{Q}{\alpha\beta\gamma K_{0}}\Big)
\Big(K_{0}^{2}+\frac{N}{\alpha\beta}\Big)^{1/2}\\&\Big(K_{0}^{2}+\frac{N}{\alpha\gamma}\Big)^{1/2}\Big(\alpha\beta\gamma\Big)^{-1/2}
\end{split}\end{equation*}
with some $1 \leq K_{0} \leq K$.  Here the summation term is bounded by
$$\min\Big(K,\frac{Q}{\alpha\beta\gamma}\Big)\frac{K + 2\sqrt{N}}{\sqrt{\alpha\beta\gamma}}+\frac{N}{\alpha\beta\gamma} \le
[\sqrt{KQ}(K + 2\sqrt{N}) + N]/\alpha\beta\gamma.$$
Hence
\begin{eqnarray*}
\SS_{2}(\AA\times\BB) & \ll & [(KQ)^{1/2} + (QNK^{-1})^{1/2} + NK^{-1}](\log{Q})^{2A +5}\\
& \ll & [QC^{-1/2} + (CN)^{1/2} + CNQ^{-1}](\log{Q})^{2A +5}.
\end{eqnarray*}
This proves the following
\begin{proposition}\label{proposition51}
Let $C = Q^{\eps}$ and $N \leq Q^{2 - 2\eps}$.  Then for any complex numbers $a_{m},b_{n}$ satisfying \eqref{24} for $1 \leq m,n \leq N$ we have
\begin{equation}\label{51}
\SS_{2}(\AA\times\BB)  \ll Q^{1 - \eps/3}.
\end{equation}
\end{proposition}

\section{Estimation of $\SS^{*}(\AA\times\BB)$}\label{section6}

We start from \eqref{412} with $k=1$ being omitted if $\chi = 1$ is singular.  Applying \eqref{42} this becomes
\begin{equation}\label{60}\begin{split}
\Delta^{*}(m,n) = &\sum\sum_{\substack{(ac,mn)=1\\ac \leq C}}\frac{\mu(a)\mu(ac)}{a\varphi(ac)}\sum\sum_{gh \mid (m,n)}\mu(g)
\sum_{k \leq K}\frac{1}{\varphi([a,k])}\\
&\sum\sum\sum_{\substack{\alpha\beta \mid \mathfrak{m}\\\alpha\gamma\mid \mathfrak{n}}}\mu(\alpha)\frac{\mu(\alpha\beta\gamma)}{\varphi(\alpha\beta\gamma)}
\sumflat_{(u,\alpha\beta\gamma ak)=1}\frac{1}{u\varphi(u)}
\Psi_{1}\Big(\frac{|m-n|(a,k)c}{\alpha\beta\gamma kuhQ}\Big)\sumsharp_{\chi \,(\!\bmod{k})}\chi(\mathfrak{m})\chibar(\mathfrak{n}).
\end{split}\end{equation}
Here the superscript $\sharp$ in the summation over $\chi$ indicates that the singular character is removed.  Recall the notation \eqref{38}.
Hence
\begin{equation}\label{61}\begin{split}
\SS^{*}(\AA\times\BB) = &\sum\sum_{ac \leq C}\frac{|\mu(ac)|}{a\varphi(ac)}\sum_{g}\sum_{h}
\sum_{k \leq K}\frac{1}{\varphi(k)}\sum_{\alpha}\sum_{\beta}\sum_{\gamma}
\frac{|\mu(\alpha\beta\gamma)|}{\varphi(\alpha\beta\gamma)}\sum_{u}\frac{1}{u\varphi(u)}\\
&\sumsharp_{\chi \,(\!\bmod{k})}\Big|\sum\sum_{\substack{(mn,ac)=1\\m\equiv 0 (gh\alpha\beta)\\n\equiv 0 (gh\alpha\gamma)}}
a_{m}b_{n}F(m,n)\Psi_{1}\Big(\frac{|m-n|}{H}\Big)\chi(\mathfrak{m})\chibar(\mathfrak{n})\Big|,
\end{split}\end{equation}
where $H = \alpha\beta\gamma kuhQ/(a,k)c$.  Separate $|m-n|$ and $H$ by changing the variable $t$ in \eqref{48} to $H/y$ and use the
estimate $H\{y/H\} \leq y$.
Then we put $ac = b$ getting
\begin{equation}\begin{split}\label{62}
\SS^{*}(\AA\times\BB) = &\sumflat_{b \leq C}\frac{b}{\varphi(b)^{2}}\sum_{g}\sum_{h}\sum_{\alpha}\sum_{\beta}\sum_{\gamma}
\frac{|\mu(\alpha\beta\gamma)|}{\varphi(\alpha\beta\gamma)} \int y^{-1}\sum_{k \leq K}\frac{1}{\varphi(k)}\\&
\sumsharp_{\chi \,(\!\bmod{k})}\Big|\sum\sum_{\substack{(mn,b)=1\\m\equiv 0 (gh\alpha\beta)\\n\equiv 0 (gh\alpha\gamma)}}
a_{m}b_{n}F(m,n)\Omega\Big(\frac{|m-n|}{y}\Big)\chi(\mathfrak{m})\chibar(\mathfrak{n})\Big|.
\end{split}\end{equation}
Note and remember that $y$ runs over a segment $1 \ll y \ll N$.

Put $d_{1} = gh\alpha\beta, d_{2} = gh\alpha\gamma$, so $[d_{1},d_{2}] = gh\alpha\beta\gamma$.  The inner double sum in \eqref{62} is equal to
$\chi(\beta)\chibar(\gamma)$ times the sum
\begin{equation}\label{63}
\SS_{\chi}(d_{1},d_{2}) = \sum\sum_{\substack{(mn,b)=1\\m\equiv 0 (d_{1})\\n\equiv 0 (d_{2})}}
a_{m}b_{n}F(m,n)\Omega\Big(\frac{|m-n|}{y}\Big)\chi\Big(\frac{m}{d_{1}}\Big)\chibar\Big(\frac{n}{d_{2}}\Big).
\end{equation}
Keep in mind that $\SS_{\chi}(d_{1},d_{2})$ depends also on $b$ and $y$.  In the next section we are going to estimate $\SS_{\chi}(d_{1},d_{2})$
by using the Hybrid Large Sieve Inequality \eqref{16}.  Applying \eqref{71} to \eqref{62} we conclude the following estimates for
$\SS^{*}(\AA\times\BB)$.
\begin{proposition}\label{proposition61}
Suppose $\AA = (a_{m})$ is given by \eqref{25} with any complex numbers $\rho(r)$ for $1 \leq r \leq X$ satisfying \eqref{220}.
Moreover, suppose the $L$-function \eqref{26} has Euler product of degree $g\leq 3$.  Assume similar conditions for $\BB = (b_{n})$.  Then
\begin{equation}\label{64}
\SS^{*}(\AA\times\BB) \ll (QC^{-1} + X)(\log{Q})^{A}
\end{equation}
if $g \leq 2$, and \eqref{64} holds for $g=3$ but with $X$ replaced by $N^{1/3}X^{2/3}$.
\end{proposition}

Note that if $g \leq 2$ then the bound \eqref{61} with $C = Q^{\eps}$ and $X = Q^{1-\eps}$ yields
\begin{equation}\label{65}
\SS^{*}(\AA\times\BB) \ll Q^{1-\eps/2}.
\end{equation}

Next consider an $L$-function of degree three mollified by a Dirichlet polynomial of length $X$.  The twisted $L$-function
$\mathcal{L}(s,\chi)$ with $\chi \pmod{q}$, $q \asymp Q$, can be approximated by two Dirichlet polynomials of length $\asymp Q^{3/2}$.
Therefore our sequences $\AA,\BB$ have length $\ll N = XQ^{3/2}$.  Then the bound \eqref{64} with $X$ replaced by $N^{1/3}X^{2/3} = XQ^{1/2}$
yields \eqref{65} provided
\begin{equation}\label{66}
X \ll Q^{1/2-\eps}.
\end{equation}

\section{Estimation of $\SS_{\chi}(d_{1},d_{2})$}\label{section7}

In this section we assume that $\AA = (a_{m})$ is given by \eqref{25} with any complex numbers $\rho(r)$ for $1 \leq r \leq X$ satisfying \eqref{220}.
Moreover, we assume that the $L$-function \eqref{26} has Euler product of degree $g \leq 3$ and no character is singular for $\mathcal{L}(s)$ except
possibly the trivial character $\chi =1$.  We also assume similar conditions for $\BB = (b_{n})$.
\begin{lemma}\label{lemma71}
For any $d_{1},d_{2} \geq 1$, $K \geq 1$ and $g=1,2$, we have
\begin{equation}\label{71}
\sum_{k \leq K}\frac{1}{\varphi(k)}\sumsharp_{\chi \,(\!\bmod{k})}|\SS_{\chi}(d_{1},d_{2})| \ll
\frac{\tau(d_{1}d_{2})^{2g}}{\sqrt{d_{1}d_{2}}}(K+X)(\log{Q})^{A}.
\end{equation}
This estimate also holds if $g=3$, but with $X$ replaced by $X^{1/3}N^{2/3}$.  Here the superscript $\sharp$ indicates that the singular character
is removed.
\end{lemma}
\begin{proof}
We write
$$F(m,n)\Omega\Big(\frac{|m-n|}{y}\Big) = \oint\!\!\oint f(s_{1},s_{2})m^{-s_{1}}n^{-s_{2}}ds_{1}ds_{2}$$
with
$$f(s_{1},s_{2}) = \iint F(u,v)\Omega\Big(\frac{|u-v|}{y}\Big)u^{s_{1}-1}v^{s_{2}-1}dudv.$$
Change the variables $u = v(1+w)$ getting
$$f(s_{1},s_{2}) = \iint F(v+vw,v)\Omega\Big(\frac{v|w|}{y}\Big)(1+w)^{s_{1}-1}v^{s_{1}+s_{2}-1}dvdw.$$
Note that $|w| \asymp y/v$ and $N^{-1} \leq 1 + w \leq N$.
Let $0 \leq \sigma_{1},\sigma_{2} \leq 1$.  Estimating trivially we get
\begin{equation}\label{73}
f(s_{1},s_{2}) \ll N^{\sigma_{1}+\sigma_{2}}(\log{N})^{2}.
\end{equation}
We shall do better by partial integration.  Integrating $a \geq 1$ times ($a =1,2$ only because of the restriction in \eqref{23}) with respect
to $w$ we get
$$f(s_{1},s_{2}) = \frac{(-1)^{a}}{s_{1}\cdots(s_{1}+a-1)}\iint\Big(F(v(1+w),v)\Omega\Big(\frac{v|w|}{y}\Big)\Big)^{(a)}(1+w)^{s_{1}+a-1}
v^{s_{1}+s_{2}-1}dvdw$$
where
$$\Big(F(\cdot)\Omega(\cdot)\Big)^{(a)} = v^{a}\frac{\partial^{a}}{\partial u^{a}}F(\cdot)\Omega(\cdot)
+\cdots + \Big(\frac{v}{y}\Big)^{a}F(\cdot)\Omega^{(a)}(\cdot)\ll (1+w)^{-a}+|w|^{-a}.$$
Hence
\begin{eqnarray*}
f(s_{1},s_{2}) &\ll& |s_{1}|^{-a}\int_{1}^{N}\int[(1+w)^{\sigma_{1}-1} + (1+w)^{\sigma_{1}+a-1}|w|^{-a}]v^{\sigma_{1}+\sigma_{2}-1}dvdw\\
& \ll & |s_{1}|^{-a}\int_{1}^{N}\Big(1+\frac{y}{v}\Big)^{\sigma_{1}}[\log{N} + \Big(1+\frac{v}{y}\Big)^{a-1}]v^{\sigma_{1}+\sigma_{2}-1}dv\\
& \ll & |s_{1}|^{-a}(N+y)^{\sigma_{1}}[\log{N} + \Big(1+\frac{N}{y}\Big)^{a-1}]N^{\sigma_{2}}\log{N}\\
& \ll & |s_{1}|^{-1}\Big(\frac{N}{|s_{1}|y}\Big)^{a-1}N^{\sigma_{1}+\sigma_{2}}(\log{N})^{2}.
\end{eqnarray*}
By symmetry this bound holds with $|s_{1}|$ replaced by $|s_{2}|$.  Combining the result with the trivial bound \eqref{73} we get
\begin{equation}\label{74}
f(s_{1},s_{2}) \ll (1 + |s_{1}| + |s_{2}|)^{-1}\Big(1+(|s_{1}|+|s_{2}|)\frac{y}{N}\Big)^{1-a}N^{\sigma_{1}+\sigma_{2}}(\log{N})^{2}.
\end{equation}
Next, integrating $b \geq 1$ times ($b=1,2$ only because of the restriction in \eqref{23}) with respect to $v$ we get
\begin{equation*}\begin{split}
f(s_{1},s_{2}) = \frac{(-1)^{b}}{(s_{1}+s_{2})\cdots(s_{1}+s_{2}+b-1)}\iint&\Big(F(v(1+w),v)\Omega\Big(\frac{v|w|}{y}\Big)\Big)^{(b)}\\
&(1+w)^{s_{1}-1}v^{s_{1}+s_{2}+b-1}dvdw
\end{split}\end{equation*}
where
$$\Big(F(\cdot)\Omega(\cdot)\Big)^{(b)} \ll v^{-b} + \Big(\frac{|w|}{y}\Big)^{b} \ll v^{-b}.$$
Hence
\begin{eqnarray*}
f(s_{1},s_{2}) & \ll & |s_{1}+s_{2}|^{-b}\int_{1}^{N}\Big(1 + \frac{y}{v}\Big)^{\sigma_{1}}v^{\sigma_{1}+\sigma_{2}-1}dv\\
& \ll & |s_{1} + s_{2}|^{-b}N^{\sigma_{1}+\sigma_{2}}\log{N}.
\end{eqnarray*}
Combining this with the trivial bound \eqref{73} we get
\begin{equation}\label{75}
f(s_{1},s_{2}) \ll (1 + |s_{1}+s_{2}|)^{-b}N^{\sigma_{1}+\sigma_{2}}(\log{N})^{2}.
\end{equation}

Finally, combining \eqref{74} and \eqref{75} with $a=b=2$ we deduce the following estimate for the Mellin transform
\begin{equation}\label{76}
f(s_{1},s_{2}) \ll (1 + |s_{1}|+|s_{2}|)^{-1}(1+|s_{1}+s_{2}|)^{-1}\big(1+(|s_{1}|+|s_{2}|)\frac{y}{N}\big)^{-1}
N^{\sigma_{1}+\sigma_{2}}(\log{N})^{2}.
\end{equation}

By contour integration on vertical lines $\Re{s_{1}} = \sigma_{1} > 1/2$, $\Re{s_{2}} = \sigma_{2} > 1/2$,
\begin{equation}\label{77}
\SS_{\chi}(d_{1},d_{2}) = \oint\!\!\oint f(s_{1},s_{2})A_{d_{1}}(s_{1} + \frac{1}{2},\chi)B_{d_{2}}(s_{2}+\frac{1}{2},\chibar)ds_{1}ds_{2}
\end{equation}
where
\begin{equation}\label{78}
A_{d}(s,\chi) = \sum_{m \equiv 0 (d)}a_{m}\chi\Big(\frac{m}{d}\Big)m^{1/2-s}
\end{equation}
and $B_{d}(s,\chibar)$ is defined similarly.  We assume that both series have analytic continuation to the lines $\sigma_{1} = 0$, $\sigma_{2} =0$
without poles because $\chi$ is non-singular.  For $a_{m}$ given by \eqref{25} we get
\begin{eqnarray*}
A_{d}(s,\chi) & = & \sum_{rl \equiv 0 (d)}\rho(r)\lambda(l)\chi\Big(\frac{rl}{d}\Big)(rl)^{-s}\\
& = & d^{-s}\sum_{r}\rho(r)\chi\Big(\frac{r}{(d,r)}\Big)\Big(\frac{(d,r)}{r}\Big)^{s}
\sum_{l=1}^{\infty}\lambda\Big(\frac{dl}{(d,r)}\Big)\chi(l)l^{-s}.
\end{eqnarray*}
\begin{lemma}\label{lemma72}
For $\Re{s} \geq 0$ we have
$$\sum_{l=1}^{\infty}\lambda(\delta l)\chi(l)l^{-s} = L(s,\chi)P_{\delta}(s,\chi)$$
where
$$P_{\delta}(s,\chi) \ll \tau(\delta)^{2g}.$$
\end{lemma}
\begin{proof}
This is an easy exercise with the Euler product
$$L(s) = \sum_{l=1}^{\infty}\lambda(l)l^{-s} = \prod_{p}(1 - \alpha_{1}(p)p^{-s})^{-1}\cdots (1- \alpha_{g}(p)p^{-s})^{-1}.$$
\end{proof}
Hence
\begin{eqnarray*}
A_{d}(s,\chi) &=& L(s,\chi)d^{-s}\sum_{r}\rho(r)\chi\Big(\frac{r}{(d,r)}\Big)\Big(\frac{(d,r)}{r}\Big)^{s}P_{d/(d,r)}(s,\chi)\\
& = & L(s,\chi)d^{-s}\sum_{\delta \mid d}P_{\delta}(s,\chi)\sum_{(r,\delta)=1}\rho\Big(\frac{d}{\delta}r\Big)\chi(r)r^{-s}.
\end{eqnarray*}
\begin{lemma}\label{lemma73}
For any complex numbers $\rho(r)$, $1 \leq r \leq X$, we have
\begin{equation*}\begin{split}
\sum_{k \leq K}\,\,\sumsharp_{\chi \,(\!\bmod{k})}\int_{-T}^{T}&\Big|\sum_{r \leq X}\rho(r)\chi(r)r^{-1/2-it}\Big|^{2}|L\big(\frac{1}{2}+it,\chi\big)|^{2}dt\\
&\ll [K^{2}T+X(KT)^{g/2}]\Big(\sum_{r \leq X}\frac{|\rho(r)|^{2}}{r}\Big)(\log{KT})^{2g}
\end{split}\end{equation*}
where $K \geq 1$, $T \geq 1$ and the superscript $\sharp$ tells us that the singular character is omitted.
\end{lemma}
\begin{proof}
This follows by standard arguments using the approximate functional equation for $L(s,\chi)$ (cf. Theorem 5.3 of \cite{IK}) and the
HLSI as stated in \eqref{16}.
\end{proof}

By Lemma \ref{lemma73} it follows that
\begin{equation}\label{79}
\sum_{k \leq K}\,\,\sumsharp_{\chi \,(\!\bmod{k})}\int_{-T}^{T}|A_{d}\big(\frac{1}{2}+it,\chi\big)|^{2}dt
\ll d^{-1}\tau(d)^{2g}[K^{2}T+X(KT)^{g/2}](\log{KTX})^{A}.
\end{equation}

By \eqref{77}, \eqref{76} we find that the left side of \eqref{71} is bounded by
$$\frac{1}{K_{0}T}\sum_{k \leq K_{0}}\,\,\sumsharp_{\chi \,(\!\bmod{k})}\int_{-T}^{T}|A_{d_{1}}(\frac{1}{2}+it,\chi)||B_{d_{2}}(\frac{1}{2}+it,\chibar)|
dt(\log{Q})^{4}$$
for some $1 \leq K_{0} \leq K$ and $T \geq 1$.  Hence by Cauchy's inequality and \eqref{78} we get
\begin{equation}\label{710}
\frac{1}{K_{0}T}\frac{\tau(d_{1}d_{2})^{2g}}{\sqrt{d_{1}d_{2}}}[K_{0}^{2}T + X(K_{0}T)^{g/2}](\log{Q})^{A}.
\end{equation}
This proves \eqref{71} if $g \leq 2$.  for $g = 3$ we combine \eqref{710} with another estimate of similar type
\begin{equation}\label{711}
\frac{1}{K_{0}T}\frac{\tau(d_{1}d_{2})^{2g}}{\sqrt{d_{1}d_{2}}}[K_{0}^{2}T + N](\log{Q})^{A}.
\end{equation}
This estimate is derived in a similar way as was  \eqref{710} except that we keep the information $m,n \leq N$ from the support of $F(m,n)$.
Then \eqref{78} is a Dirichlet polynomial of length $N$ and without appealing to any $L$-function we get \eqref{711} straight by the HLSI.  Now
we can replace the second terms in \eqref{710}, \eqref{711} by their minimum
$$\min(N,X(K_{0}T)^{3/2}) \leq N^{1/3}X^{2/3}K_{0}T.$$
This proves \eqref{71} if $g=3$.
\end{proof} 

\section{Estimation of $\SS^{+}(\AA\times\BB)$}\label{section8}

Here $\Delta^{+}(m,n)$ is given by \eqref{416} and it goes through the same transformations as $\Delta^{*}(m,n)$, so $\Delta^{+}(m,n)$ is also
given by \eqref{60} with $k=1$.  Hence $\SS^{+}(\AA\times\BB)$ satisfies \eqref{61} with $k=1$ and $H = \alpha\beta\gamma uhQ/c$.  To get
\eqref{62} from \eqref{61} we used the estimate $H\{y/H\} \leq y$, but now we need a slightly better estimate
$$H\{y/H\} \leq \min(y,H) \leq \min(y,\alpha\beta\gamma uhQ).$$
Moreover we have
$$\sum_{u}\frac{1}{u\varphi(u)}\min(y,\alpha\beta\gamma uhQ) \ll \min(y, \alpha\beta\gamma hQ)\log{Q}.$$
Hence \eqref{62} for $\SS^{+}(\AA\times\BB)$ becomes
\begin{equation*}\begin{split}
\SS^{+}(\AA\times\BB) \ll &(\log{Q})\sumflat_{b \leq C}\frac{b}{\varphi(b)^{2}}\sum_{g}\sum_{h}\sum_{\alpha}\sum_{\beta}\sum_{\gamma}
\frac{|\mu(\alpha\beta\gamma)|}{\varphi(\alpha\beta\gamma)}\\
& \int y^{-2}\min(y, \alpha\beta\gamma hQ)\Big|\sum\sum_{\substack{(mn,b)=1\\m\equiv 0 (gh\alpha\beta)\\n\equiv 0 (gh\alpha\gamma)}}
a_{m}b_{n}F(m,n)\Omega\Big(\frac{|m-n|}{y}\Big)\Big|dy.
\end{split}\end{equation*}
The co-primality conditions $(m,b)=1$ and $(n,b)=1$ can be resolved separately by M\"{o}bius formula.  The resulting divisors of $(m,b)$
and $(n,b)$ can be attached to $\beta$ and $\gamma$ respectively, except for their greatest common factor which can be attached to $\alpha$.
This way we arrive at the following estimate
\begin{equation*}\begin{split}
\SS^{+}(\AA\times\BB) \ll &(\log{Q})^{2}\sum_{g}\sum_{h}\sum_{\alpha}\sum_{\beta}\sum_{\gamma}
\frac{|\mu(\alpha\beta\gamma)|}{\varphi(\alpha\beta\gamma)}\\
& \int y^{-2}\min(y, \alpha\beta\gamma hQ)\Big|\sum\sum_{\substack{m\equiv 0 (gh\alpha\beta)\\n\equiv 0 (gh\alpha\gamma)}}
a_{m}b_{n}F(m,n)\Omega\Big(\frac{|m-n|}{y}\Big)\Big|dy.
\end{split}\end{equation*}
We make further simplifications by putting $gh\alpha = d$.  We get
\begin{equation}\begin{split}\label{81}
\SS^{+}(\AA\times\BB) \ll (\log{Q})^{3}&\sum_{d}\tau(d)\sum\sum_{(\beta,\gamma)=1}\frac{1}{\beta\gamma}\int y^{-2}\min(y, \beta\gamma dQ)\\
&\Big|\sum\sum_{\substack{m\equiv 0 (d\beta)\\n\equiv 0 (d\gamma)}}a_{m}b_{n}F(m,n)\Omega\Big(\frac{|m-n|}{y}\Big)\Big|dy.
\end{split}\end{equation}
Note and remember that $y$ runs over a segment $1 \ll y \ll N$.

Using the bound $\beta\gamma dQ$ for the minimum in \eqref{81} and estimating the double inner sum trivially, one could show that
$\SS^{+}(\AA\times\BB) \ll Q(\log{Q})^{A}$, while our goal is
$$\SS^{+}(\AA\times\BB) \ll Q(\log{Q})^{-C}$$
for any $C\geq 0$.  Therefore we need to save only a little bit from the trivial estimation.  This saving will come from cancellation in the
coefficients $a_{m},b_{n}$.  It is just for this purpose that we now assume these coefficients to be of the form \eqref{25} with
\begin{equation}\label{82}
\lambda(l) = 1
\end{equation}
and $\rho(r)$ satisfying \eqref{224}.  In order to be able to use these special features we first need to eliminate the contribution of large
divisors $d\beta$ and $d\gamma$, because they may take so much out of $m$ and $n$ respectively that there is nothing left in $l$ or $r$ to
play with.

We split the right side of \eqref{81} into two parts according to
\begin{equation}\label{83}
\beta\gamma dQ > yQ^{\eps}
\end{equation}
\begin{equation}\label{84}
\beta\gamma dQ \leq yQ^{\eps}
\end{equation}
and we estimate $\min(y, \beta\gamma dQ)$ by $y$ or $\beta\gamma dQ$ in the ranges \eqref{83} or \eqref{84}, respectively.  We get
\begin{equation}\label{85}
\SS^{+}(\AA\times\BB) \ll (S_{1} + S_{2})(\log{Q})^{3}
\end{equation}
say, where
$$S_{1} = \int\sum\sum\sum_{\beta\gamma dQ > yQ^{\eps}}\frac{\tau(d)}{\beta\gamma}
\Big|\sum\sum_{\substack{m\equiv 0 (d\beta)\\n\equiv 0 (d\gamma)}}a_{m}b_{n}F(m,n)\Omega\Big(\frac{|m-n|}{y}\Big)\Big|\frac{dy}{y},$$
$$S_{2} = Q\int\sum\sum\sum_{\substack{\beta\gamma dQ \leq yQ^{\eps}\\(\beta,\gamma)=1}}\tau(d)d\int
\Big|\sum\sum_{\substack{m\equiv 0 (d\beta)\\n\equiv 0 (d\gamma)}}a_{m}b_{n}F(m,n)\Omega\Big(\frac{|m-n|}{y}\Big)\Big|\frac{dy}{y^{2}}.$$

We estimate $S_{1}$ using only the crude bounds \eqref{24} for the coefficients $a_{m}$,$b_{n}$.  We get
\begin{eqnarray}\label{86}\nonumber
S_{1} & \ll & \sum_{\beta \leq N}\sum_{\gamma \leq N}\sum_{d \leq N}\frac{\tau(d)}{\beta\gamma}
\sum\sum_{\substack{m,n \leq N\\d\beta \mid m,\,d\gamma \mid n\\|m-n| < \beta\gamma dQ^{1-\eps}}}\frac{\tau(m)^{A}\tau(n)^{A}}
{\sqrt{mn}}\\
\nonumber & \ll & \sum_{\beta \leq N}\sum_{\gamma \leq N}\sum_{d \leq N}\frac{\tau(\beta\gamma d)^{A}}{(\beta\gamma)^{3/2}d}
\sum\sum_{\substack{m,n \leq N\\|\frac{m}{\gamma} - \frac{n}{\beta}| < Q^{1-\eps}}}\Big[\Big(\frac{\gamma}{\beta}\Big)^{1/2}
\frac{\tau(m)^{2A}}{m}+\Big(\frac{\beta}{\gamma}\Big)^{1/2}\frac{\tau(n)^{2A}}{n}\Big]\\
& \ll & \sum_{\beta \leq N}\sum_{\gamma \leq N}\sum_{d \leq N}\frac{\tau(\beta\gamma d)^{A}}{\beta\gamma d}Q^{1-\eps}(\log{Q})^{A}
\ll Q^{1-\eps}(\log{Q})^{A}.
\end{eqnarray}

Estimation of $S_{2}$ requires extra conditions for the coefficients $a_{m}$,$b_{n}$ and the test function $F(m,n)$.  We assume the these
coefficients are of the form \eqref{25} with $\lambda(l)=1$, that is with \eqref{26} being the zeta function $\zeta(s)$, and that $\rho(r)$
satisfy \eqref{224}.  Actually the $\rho(r)$ in $a_{m}$ and those in $b_{n}$ can be different, say they are $\rho_{A}(r)$, $\rho_{B}(r)$ respectively.
Assume they are supported in dyadic segments $X_{A} \leq r \leq 2X_{A}$, $X_{B} \leq r \leq 2X_{B}$ respectively, with
\begin{equation}\label{87}
1 \leq X_{A},X_{B} \leq X \leq Q^{1 - 2\delta}.
\end{equation}
So far we have been assuming that the support of $F(m,n)$ is in the box \eqref{22}.  Now we need to localize the variables a little more.  We assume,
in addition to the former properties, that the support of $F(m,n)$ implies
\begin{equation}\label{88}
Q^{-\delta} \leq \frac{m}{n} \leq Q^{\delta}, \qquad mn \leq X_{A}X_{B}Q^{2 - 2\delta}
\end{equation}
for some constant $\delta > 0$ (in our applications any $0 < \delta < 1/2$ will be sufficient).  Moreover we extend \eqref{23} to
\begin{equation}\label{88'}
x^{i}y^{j}F^{(ij)}(x,y) \ll 1
\end{equation}
for any $i,j \geq 0$, the implied constant depending on $i,j$.

First we are going to evaluate the inner double sum
\begin{equation}\label{89}
V_{d_{1}d_{2}}(y) = \sum\sum_{\substack{m \equiv 0 (d_{1})\\n \equiv 0 (d_{2})}}a_{m}b_{n}F(m,n)\Omega\Big(\frac{|m-n|}{y}\Big).
\end{equation}
By \eqref{25}, $m = lr \equiv 0 \pmod{d}$ means $l = d'l'$ with $d' = d/(d,r)$.  Therefore \eqref{89} becomes
$$V_{d_{1}d_{2}}(y) = \sum_{r_{1}}\sum_{r_{2}}\frac{\rho_{A}(r_{1})\rho_{B}(r_{2})}{\sqrt{[r_{1},d_{1}][r_{2},d_{2}]}}
\sum_{l_{1}}\sum_{l_{2}}\frac{F([r_{1},d_{1}]l_{1},[r_{2},d_{2}]l_{2})}{\sqrt{l_{1}l_{2}}}\Omega\Big(\frac{|\cdot|}{y}\Big).$$
Here $l_{1},l_{2}$ run over positive integers, free of any arithmetical constraints but with smooth weights given by $F$ and $\Omega$.
From the support of $\Omega(|\cdot|/y)$ we see that $|m-n| \asymp y$ which range translates to
\begin{equation}\label{810}
|[r_{1},d_{1}]l_{1} - [r_{2},d_{2}]l_{2}| \asymp y.
\end{equation}
On the ther hand \eqref{84} tells us that $y$ is not very small, precisely for $d_{1} = d\beta$, $d_{2} = d\gamma$ we get
$$y \geq [d_{1},d_{2}]Q^{1-\eps} \geq d_{1}Q^{1-\eps} \geq d_{1}r_{1}X^{-1}Q^{1-\eps} \geq [d_{1},r_{1}]Q^{2\delta - \eps}.$$
Similarly $y \geq [d_{2},r_{2}]Q^{2\delta - \eps}$.  Therefore the integers $l_{1},l_{2}$ have considerable room to run over intervals
of length at least $Q^{2\delta - \eps}$.
Moreover, by $|m-n| \asymp y$ combined with the first condition in \eqref{88} it follows that $m,n \gg yQ^{-\delta}$.  For
$m = [r_{1},d_{1}]l_{1}$ this gives
$$d_{1}l_{1}Q^{1-2\delta} \gg r_{1}d_{1}l_{1} \gg [r_{1},d_{1}]l_{1} \gg [d_{1},d_{2}]Q^{1-\delta-\eps} \geq d_{1}Q^{1-\delta-\eps}$$
by \eqref{87}, hence $l_{1} \gg Q^{\delta - \eps}$.  Similarly we show that $l_{2} \gg Q^{\delta - \eps}$.  Since $l_{1}$,$l_{2}$ are
weighted smoothly, it allows us to replace the summation by the corresponding integration with a small error term, smaller than any negative power
of $Q$.  Next $l_{1}$,$l_{2}$ being continuous variables we change them by factors $[r_{1},d_{1}]$, $[r_{2},d_{2}]$ respetively getting
$$V_{d_{1}d_{2}}(y) = \Big(\sum_{r}\frac{\rho_{A}(r)}{[r,d_{1}]}\Big)\Big(\sum_{r}\frac{\rho_{B}(r)}{[r,d_{2}]}\Big)
\iint \frac{F(u,v)}{\sqrt{uv}}\Omega\Big(\frac{|u-v|}{y}\Big)dudv +O(Q^{-1}).$$

We have not yet exploited the second condition in \eqref{88} which says $uv \leq X_{A}X_{B}Q^{2 - 2\delta}$.  On the other hand we have
$|u-v|\asymp y$ and $Q^{-\delta} \leq u/v \leq Q^{\delta}$ which imply $uv \gg y^{2}Q^{-\delta}$.  Hence
$$X_{A}X_{B} \gg y^{2}Q^{\delta-2} > [d_{1},d_{2}]^{2}Q^{\delta - 2\eps} \geq d_{1}d_{2}Q^{\delta-2\eps},$$
so either $X_{A} \geq d_{1}Q^{\delta/2 - \eps}$ or $X_{B} \geq d_{2}Q^{\delta/2 -\eps}$.  This shows there is more than $Q^{\delta/2 - \eps}$
room for summation over $r$ in one of the above sums, and due to the assumption \eqref{224} we gain a factor $(\log{Q})^{-C}$ relative to a
trivial estimate.
The double integral in $u,v$ is easily estimated by $O(y\log{N})$.  Therefore
\begin{eqnarray*}
V_{d_{1}d_{2}}(y) \ll (\log{Q})^{1-C}y\Big(\sum_{r \leq X}\frac{\tau(r)^{A}}{[r,d_{1}]}\Big)\Big(\sum_{r \leq X}\frac{\tau(r)^{A}}{[r,d_{2}]}\Big) \ll y\frac{\tau(d_{1}d_{2})^{A}}{d_{1}d_{2}}(\log{Q})^{-C}.
\end{eqnarray*}
Hence
\begin{equation}\label{811}
S_{2} \ll Q(\log{Q})^{-C}\sum_{d_{1}<N}\sum_{d_{2}<N}\frac{\tau(d_{1}d_{2})^{A}}{d_{1}d_{2}}(d_{1},d_{2}) \ll Q(\log{Q})^{-C}.
\end{equation}

Finally, adding \eqref{811} to \eqref{86} we conclude by \eqref{85}
\begin{proposition}\label{proposition81}
Suppose $\AA = (a_{m})$ is given by \eqref{25} with complex numbers $\rho(r)$ supported on $X_{A} \leq r \leq 2X_{A}$, satisfying \eqref{220}
and \eqref{224}.  Moreover suppose the $L$-function \eqref{26} is exactly the zeta function $\zeta(s)$.  Assume similar conditions for
$\BB = (b_{n})$.  Let $F(m,n)$ be a smooth function supported in the box \eqref{22} whose partial derivatives satisfy \eqref{88'}.  Finally assume
the support of $F(m,n)$ implies the restrictions \eqref{88} with some small constant $\delta >0$.  Then
\begin{equation}\label{812}
\SS^{+}(\AA\times\BB) \ll Q(\log{Q})^{-C}
\end{equation}
with any $C \geq 0$, the implied constant depending on $C$.
\end{proposition}

Adding \eqref{812} to \eqref{65} we get
\begin{equation}\label{813}
\SS_{1}(\AA\times\BB) \ll Q(\log{Q})^{-C}.
\end{equation}
Then adding \eqref{813} to \eqref{51} we complete (by Lemma \ref{lemma32} and Lemma \ref{lemma45}) the proof of Theorem \ref{theorem24}.

\section{Modifications of the test function}\label{section9}

The conditions for the test function $F(m,n)$ in the bilinear form $\SS(\AA\times\BB)$ are robust, but not quite flexible for some applications.
As it is (in basic Theorems \ref{theorem22}, \ref{theorem23}, \ref{theorem24}) $F(m,n)$ depends on the variables $m$,$n$ and so it does not capture
the factorization properties of our coefficients $a_{m}$,$b_{n}$ given by \eqref{25}.  To bring the results a bit closer to applications environment
one should consider test functions of type
\begin{equation}\label{91}
F(l_{1},r_{1};l_{2},r_{2};q)
\end{equation}
where $l_{1}r_{1} = m$, $l_{2}r_{2} = n$ are our original variables and $q$ runs over the moduli of our family of characters, so the restrictions
\eqref{22} imply $l_{1}r_{1} \leq N$, $l_{2}r_{2} \leq N$ and $q \asymp Q$, while the ones in \eqref{88} imply
\begin{equation}\label{92}
Q^{-\delta} < \frac{l_{1}r_{1}}{l_{2}r_{2}} < Q^{\delta}, \qquad l_{1}l_{2} \ll Q^{2 - 2\delta}.
\end{equation}
Note that we do not need precise restrictions for $l_{1}$ and $l_{2}$ separately, but only for the product $l_{1}l_{2}$.

Assuming that the function \eqref{91} is smooth and has sufficiently many partial derivatives relatively small, one can reduce this case (by separation of
variables techniques) to the one already considered with $F(m,n)$.  A little contamination in the separation process is tolerable by the quite
flexible conditions on our coefficients $a_{m}$,$b_{n}$, while the fact that the bilinear form $\SS(\AA\times\BB)$ is linear in $F$ makes the process
straightforward.

One does not even need to assume that \eqref{91} is compactly supported.  If some of the variables $l_{1},r_{1},l_{2},r_{2}$ exceed the range of our
previous results, then \eqref{91} is very small in size and a direct, crude application of the LSI in this excessive range gives a sufficiently
strong estimate to be neglected by comparison with the main term.  We leave out the details how exactly the desirable adjustments are executed
for experienced readers.

We are going to reformulate our basic theorems for the bilinear form $\SS(\AA\times\BB)$ with the test function of the special shape
\begin{equation}\label{93}
F(l_{1},r_{1};l_{2},r_{2};q) = G\Big(\frac{l_{1}r_{1}}{l_{2}r_{2}},\frac{l_{1}l_{2}}{q^{g}}\Big)
\end{equation}
where $G(x,y)$ is a nice smooth function, rapidly decaying to zero as $x \rightarrow 0$, or $x \rightarrow \infty$, or $y \rightarrow \infty$.
This means that the main activity happens in the range $l_{1}r_{1} \asymp l_{2}r_{2}$, $l_{1}l_{2} \ll Q^{g}$.  We can write
\begin{equation}\label{94}
G\Big(\frac{l_{1}r_{1}}{l_{2}r_{2}},\frac{l_{1}l_{2}}{r_{1}r_{2}q^{g}}\Big) = G\Big(\frac{m}{n},\frac{mn}{r_{1}r_{2}q^{g}}\Big)
\end{equation}
where $m = l_{1}r_{1}$, $n=l_{2}r_{2}$ are our original variables.  Hence it suffices to perform the separation arguments only with
respect to the single variable $y$.

Suppose $G(x,y)$ is smooth on $\R^{+}\times\R^{+}$ with partial derivatives satisfying
\begin{equation}\label{95}
x^{a}y^{b}G^{(a,b)}(x,y) \ll (1 + |\log{x}|)^{-c}(1 + y)^{-c}(\log{Q})^{aA}
\end{equation}
for some $A \geq 1$ and any $a,b,c \geq 0$, the implied constant depending on $a,b,c$.
\begin{theorem25}\label{theorem25}
Assuming the conditions of Theorems \ref{theorem22}, \ref{theorem23}, \ref{theorem24} adjusted to the context of the test function
$G(x,y)$ we have
\begin{equation}\label{96}
\SS(\AA\times\BB) = \SS_{diag}(\AA\times\BB) + O(Q(\log{Q})^{-C})
\end{equation}
for any $C \geq 0$.  Here the leading term is defined by \eqref{211} and it is also given by the adjusted form of \eqref{218} which becomes
\begin{equation}\begin{split}\label{97}
\SS_{diag}(\AA\times\BB) =& \mathfrak{S}Q\sum\sum\sum\sum_{r_{1}l_{1} = r_{2}l_{2}}\rho_{A}(r_{1})\rho_{B}(r_{2})\lambda(l_{1})\lambda(l_{2})\\
& \frac{1}{r_{1}l_{1}}\prod_{p \mid r_{1}l_{1}}\Big(1 - \frac{1}{p}\Big)\Big(1 - \frac{1}{p^{2}} - \frac{1}{p^{3}}\Big)^{-1}
\int\Psi(t)G(1,l_{1}l_{2}(tQ)^{-g})dt + O(Q^{1/2 + \eps}).
\end{split}\end{equation}
\end{theorem25}


\begin{thebibliography}{99}

\bibitem[CIS]{CIS}
{ J.B.\,Conrey, H.\,Iwaniec and K.\,Soundararajan},
\textit{Critical Zeros of Dirichlet L-functions}, preprint.

\bibitem[IK]{IK}
{ H.\,Iwaniec and E.\,Kowalski},
\textit{Analytic number theory},
AMS Colloquium Publications, vol 53, AMS, Providence, RI, 2004.

\bibitem[Lin]{Lin}
{Yu.V.\,Linnik} \textit{The large sieve}, Dokl. Akad. Nauk SSSR 30 (1941), 292 -- 294.


\end{thebibliography}
\end{document}